\documentclass[11pt,a4paper]{article}

\usepackage[T1]{fontenc}
\usepackage[utf8]{inputenc}
\usepackage[english]{babel}
\usepackage{lmodern}
\usepackage{amsmath,amssymb,amsthm}
\usepackage{geometry}
\usepackage{microtype}
\usepackage{cite}
\usepackage[hidelinks]{hyperref}
\hypersetup{
  pdftitle={Frobenius orbit codes and finite-field Fourier spectra: exact distance and coset distributions},
  pdfauthor={David Kumallagov, Daniil Sizikov, Anton Zarubin},
  pdfkeywords={finite-field Fourier transform, subfield-linear code,
    Frobenius orbit, covering radius, coset leader, list decoding}
}

\newtheorem{theorem}{Theorem}[section]
\newtheorem{corollary}[theorem]{Corollary}
\theoremstyle{definition}
\newtheorem{definition}[theorem]{Definition}
\newtheorem{example}[theorem]{Example}
\theoremstyle{remark}
\newtheorem{remark}[theorem]{Remark}

\newcommand{\K}{\mathbb F_q}
\newcommand{\Lfield}{\mathbb F_Q}
\newcommand{\F}[1]{\mathbb F_{#1}}
\newcommand{\Cpi}{\mathcal C(\pi)}
\newcommand{\CA}{\mathcal S_A}
\newcommand{\Fix}{\operatorname{Fix}}
\newcommand{\wtH}{\operatorname{wt}_{\mathrm H}}
\newcommand{\dist}{\operatorname{d}_{\mathrm H}}

\title{Frobenius orbit codes and finite-field Fourier spectra:\\
exact distance and coset distributions}

\author{
David Kumallagov
\and
Daniil Sizikov
\and
Anton Zarubin}

\date{}

\begin{document}
\maketitle

\begin{abstract}
Let $q$ be a prime power, let $m\ge1$, put $Q=q^m$, and let a permutation
$\pi$ of a finite index set $I$ satisfy $\pi^m=1$.  We study the
$\F q$-linear code
\[
 \mathcal C(\pi)=\{x\in\Lfield^I:x_{\pi(i)}=x_i^q\},
\]
which includes the Fourier spectra of $\K$-valued functions on finite
abelian groups split by $\Lfield$.  Our structural starting point is a
$\F q$-linear Hamming isometry taking each cycle of $\pi$ of length $\ell$
to the subfield repetition code
$\{(b,\ldots,b):b\in\F{q^\ell}\}$.  From this normal form we obtain exact
symbol weight and received-word distance enumerators, all list sizes, and the
covering radius.  An occupancy generating function gives the complete
coset-leader distribution, the exact mean distance, the probability of a
unique nearest word, and the number of deep holes.  For the full cyclic
family of length $q^m-1$, the deficit of a uniformly random ambient word from
the covering radius is asymptotically Poisson with mean $(m-1)/2$, with an
additional $1$ when $m$ is even; a central limit theorem follows.  We also
determine the deep-hole probability and the second-order gap to the
sphere-covering bound.  Classical Fourier--Galois descent and cyclotomic
orbit enumeration enter only to identify the Fourier specialization.
\end{abstract}



\section{Introduction}

Let $q$ be a prime power, let $m\ge1$, and put $Q=q^m$.  If a finite-field discrete Fourier
transform is evaluated in $\Lfield$ but its input lies in $\K$, then its
spectral coordinates satisfy the classical conjugacy constraint
\begin{equation}
  V(\chi^q)=V(\chi)^q.                         \label{eq:intro-conjugacy}
\end{equation}
For a cyclic transform this reads $V_{qs}=V_s^q$.  This condition, its
description by $q$-cyclotomic classes, and the storage of one native-subfield
representative per Frobenius orbit are standard; see Blahut
\cite[Theorem~3]{Blahut1979}, the Frobenius FFT of van der Hoeven and
Larrieu~\cite{HoevenLarrieu}, and its additive analogue~\cite{LiEtAl2018}.
Transform-domain descriptions of abelian and subfield-linear codes are also
well established~\cite{Tanner1988,RajanSiddiqi1992,DeyRajan2005,Huffman2013}.

The Fourier conjugacy condition and its parametrization by cyclotomic orbits
are classical.  We study instead the exact nearest-word and coset geometry of
the resulting subfield-linear spectral space inside the full extension-field
ambient space.  The natural level of generality is the following semilinear
permutation code.  If $I$ is finite and
$\pi^m=1$, put
\[
  \Cpi=\{x\in\Lfield^I:x_{\pi(i)}=x_i^q\text{ for every }i\in I\}.
\]
The Fourier case is obtained by taking $I=\widehat A$ and
$\pi(\chi)=\chi^q$ for a finite abelian group $A$ split by $\Lfield$.

The central structural observation is that a coordinatewise inverse
Frobenius map sends the constituent on a cycle of length $\ell$ to
\[
 \Delta_\ell(\F{q^\ell})
 =\{(b,\ldots,b):b\in\F{q^\ell}\}\subseteq\Lfield^\ell.
\]
Thus $\Cpi$ is a direct product of subfield repetition codes up to a
$\K$-linear symbol-Hamming isometry.  The resulting restricted occupancy
model is explicit enough to determine the complete received-word distance
enumerator, all list sizes, the covering radius, every coset-leader weight,
the deep holes, and the probability of unique nearest-word decoding.

There is also an important limitation.  In the Fourier case the trivial
character is a one-element Frobenius orbit.  Consequently, the global symbol
minimum distance is $1$.  Accordingly, the full space is naturally viewed as
a nearest-word approximation space rather than as a worst-case
error-correcting code.

The Frobenius-coordinate isometry is an elementary reduction to
restricted-alphabet occupancy while retaining the native subfields on the
individual cycles.  The results established here have two parts.  First, we use this model
to derive exact distance, list, coset-leader, uniqueness, and deep-hole
formulas.  Second, for the full cyclic family we prove a Poisson approximation
and a central limit theorem for the random distance deficit, obtain a sharp
deep-hole probability, and determine the second-order gap between the exact
covering radius and the rate-only sphere-covering bound.  The closed occupancy
formulas are specific to this setting; to the best of our knowledge, the
probabilistic and second-order asymptotics are new.  Fourier--Galois
descent, M\"obius orbit enumeration, and the sphere-covering inequality
are used as standard tools.  All results are unconditional under their stated
hypotheses; no conjecture is used.  The Fourier specialization assumes the
splitting condition that the group exponent divides $q^m-1$, while all
asymptotic statements take $q$ fixed, $m\to\infty$, and $n=q^m-1$.

A general framework for coset-leader and list weight enumerators was developed
by Jurrius and Pellikaan~\cite{JurriusPellikaan2015}.  The formulas below are
more specialized: the admissible symbols form the subfield
$\F{q^\ell}\subseteq\F{q^m}$ on a cycle of length $\ell$, and the complete
enumerator factors over Frobenius cycles.  This restricted-alphabet structure
is what makes the closed occupancy formulas and the ensuing asymptotics
possible.

Subfield-linear codes with the ordinary symbol Hamming metric are equivalently
linear codes in the folded Hamming metric; we use the terminology and viewpoint
of~\cite{MartinezPenasRodriguez2025}.  Standard references
for finite fields, finite Fourier analysis, coding theory, and covering codes
are~\cite{LidlNiederreiter,Terras,HuffmanPless,CohenEtAl}.

Section~\ref{sec:semilinear} gives the diagonal normal form and basic
parameters.  Section~\ref{sec:fourier} records the Fourier specialization and
the required orbit counts.  Section~\ref{sec:geometry} develops the exact
distance and coset geometry.  Section~\ref{sec:cyclic} treats the full cyclic
family, including its covering and probabilistic asymptotics.

\section{Semilinear permutation codes}
\label{sec:semilinear}

Throughout, $q=p^a$ is a prime power, $m\ge1$, and
$Q=q^m$.  Let $I$ be a nonempty finite set of
cardinality $N$, and let $\pi\in\operatorname{Sym}(I)$ satisfy $\pi^m=1$.
Every cycle length of $\pi$ therefore divides $m$.  We equip $\Lfield^I$ with
the symbol Hamming weight and distance
\[
 \wtH(x)=|\{i:x_i\ne0\}|,
 \qquad
 \dist(x,y)=\wtH(x-y).
\]

Let $\mathfrak O(\pi)$ be the set of cycles of $\pi$.  For each
$O\in\mathfrak O(\pi)$ choose a leader $i_O$ and write
\[
 O=(i_{O,0},\ldots,i_{O,\ell_O-1}),
 \qquad i_{O,t}=\pi^t(i_O),
 \qquad \ell_O=|O|.
\]
Put $S_O=\F{q^{\ell_O}}$, viewed as the unique subfield of $\Lfield$ with
$q^{\ell_O}$ elements.

\begin{definition}
The \emph{Frobenius permutation code} associated with $\pi$ is
\begin{equation}
 \Cpi=\{x\in\Lfield^I:x_{\pi(i)}=x_i^q\text{ for all }i\in I\}.
 \label{eq:frob-code}
\end{equation}
It is a $\K$-linear code of length $N$ over the alphabet $\Lfield$.
\end{definition}

\subsection{Diagonal normal form}

\begin{theorem}
\label{thm:diagonal-normal-form}
Define $\Psi:\Lfield^I\to\prod_{O\in\mathfrak O(\pi)}\Lfield^{\ell_O}$ by
\begin{equation}
  \Psi(x)_{O,t}=x_{i_{O,t}}^{q^{m-t}},
  \qquad 0\le t<\ell_O.                       \label{eq:normalization}
\end{equation}
Then $\Psi$ is a $\K$-linear Hamming isometry and
\begin{equation}
 \Psi(\Cpi)=
 \prod_{O\in\mathfrak O(\pi)}
 \Delta_{\ell_O}(S_O),
 \qquad
 \Delta_\ell(S)=\{(b,\ldots,b):b\in S\}.     \label{eq:product-repetition}
\end{equation}
In particular, the metric isomorphism type of $\Cpi$ depends only on the
multiset of cycle lengths of $\pi$.
\end{theorem}

\begin{proof}
Every map $z\mapsto z^{q^{m-t}}$ is a $\K$-linear permutation of
$\Lfield$, hence $\Psi$ is a $\K$-linear coordinatewise Hamming isometry.
If $x\in\Cpi$ and $b=x_{i_{O,0}}$, then
\[
 x_{i_{O,t}}=b^{q^t}
 \quad\text{and}\quad
 b^{q^{\ell_O}}=b.
\]
Thus $b\in S_O$, and
$\Psi(x)_{O,t}=b^{q^m}=b$ for every $t$.
Conversely, for arbitrary $b_O\in S_O$, the assignments
\[
 x_{i_{O,t}}=b_O^{q^t}
\]
are well defined around every cycle and satisfy
\eqref{eq:frob-code}.  This constructs the inverse correspondence.
\end{proof}

\begin{corollary}
\label{cor:basic-parameters}
The code $\Cpi$ has $\K$-dimension $N$ and cardinality $q^N$.  Its symbol
weight enumerator and minimum distance are
\begin{align}
 W_{\Cpi}(z)
   &=\sum_{x\in\Cpi}z^{\wtH(x)}
     =\prod_{O\in\mathfrak O(\pi)}
       \bigl(1+(q^{\ell_O}-1)z^{\ell_O}\bigr),
       \label{eq:symbol-weight-enumerator}\\
 d_{\mathrm H}(\Cpi)&=\min_{O\in\mathfrak O(\pi)}\ell_O. \label{eq:symbol-min-distance}
\end{align}
If $m>1$, then $\Cpi$ is not $\Lfield$-linear.
\end{corollary}

\begin{proof}
On a cycle of length $\ell$, the zero seed gives weight $0$, while each of
the $q^\ell-1$ nonzero seeds gives weight $\ell$.  Independence of the cycles
gives \eqref{eq:symbol-weight-enumerator}, the dimension statement, and
\eqref{eq:symbol-min-distance}.  Finally, the all-one word belongs to
$\Cpi$.  For $a\in\Lfield$, the constant word $a\mathbf1$ belongs to $\Cpi$
if and only if $a^q=a$, so multiplication by any $a\in\Lfield\setminus\K$
does not preserve the code.
\end{proof}

\section{Finite-field Fourier descent}
\label{sec:fourier}

Let $A$ be a finite abelian group of order $N$ and exponent $E$, and assume
\begin{equation}
 E\mid Q-1.                                      \label{eq:splitting-assumption}
\end{equation}
Then $p\nmid N$, $1/N\in\K$, and $\Lfield$ is a splitting field for $A$.
Write
\[
 \widehat A=\operatorname{Hom}(A,\Lfield^*)
\]
and use the Fourier conventions
\begin{equation}
 \widehat f(\chi)=\sum_{x\in A}f(x)\chi(x),
 \qquad
 f(x)=\frac1N\sum_{\chi\in\widehat A}
             \widehat f(\chi)\chi(x)^{-1}.       \label{eq:group-fourier}
\end{equation}
The Frobenius permutation of $\widehat A$ is
$\pi_q(\chi)=\chi^q$, where $\chi^q(x)=\chi(x)^q$.
Put
\[
 \CA=\{V\in\Lfield^{\widehat A}:V(\chi^q)=V(\chi)^q\}.
\]
The following descent statement is classical; we include it to fix the
Fourier conventions used below.

\begin{theorem}
\label{thm:fourier-descent}
An $\Lfield$-valued function $V$ on $\widehat A$ is the Fourier transform of
a unique $\K$-valued function on $A$ if and only if
\begin{equation}
 V(\chi^q)=V(\chi)^q
 \quad\text{for every }\chi\in\widehat A.       \label{eq:descent}
\end{equation}
\end{theorem}

\begin{proof}
If $f$ is $\K$-valued, then
\[
 \widehat f(\chi)^q
 =\sum_{x\in A}f(x)^q\chi(x)^q
 =\widehat f(\chi^q).
\]
Conversely, apply Fourier inversion and use \eqref{eq:descent}:
\[
 f(x)^q
 =\frac1N\sum_{\chi\in\widehat A}
 V(\chi^q)(\chi^q(x))^{-1}=f(x),
\]
because $\chi\mapsto\chi^q$ permutes $\widehat A$.  Thus $f(x)\in\K$.
Uniqueness follows from Fourier inversion.
\end{proof}

\begin{remark}
Theorem~\ref{thm:fourier-descent} is included to fix conventions and to connect
the code $\mathcal C(\pi_q)$ to Fourier spectra.  Its cyclic conjugacy
constraint is classical~\cite[Theorem~3]{Blahut1979}; evaluating one value per
orbit is the one-value-per-orbit principle underlying the Frobenius
FFT~\cite{HoevenLarrieu}.
The transform-domain literature on abelian and $\F q$-linear cyclic codes
must also be distinguished from the metric results below
\cite{RajanSiddiqi1992,DeyRajan2005}.
\end{remark}

All the results of Section~\ref{sec:semilinear} now apply to $\CA$ with
$I=\widehat A$ and $\pi=\pi_q$.  In particular,
\begin{equation}
 \dim_{\K}\CA=N,
 \qquad |\CA|=q^N.                              \label{eq:fourier-code-size}
\end{equation}
Because the trivial character is fixed by Frobenius,
\begin{equation}
 d_{\mathrm H}(\CA)=1.                          \label{eq:fourier-dmin-one}
\end{equation}

\subsection{Orbit lengths for finite abelian groups}

Let $\mu$ denote the number-theoretic M\"obius function.  Write
\begin{equation}
 A\simeq\mathbb Z/n_1\mathbb Z\times\cdots\times
          \mathbb Z/n_s\mathbb Z,               \label{eq:abelian-decomposition}
\end{equation}
omitting trivial factors.  Empty products below equal $1$, and we use
$\gcd(n,0)=n$.

\begin{theorem}
\label{thm:abelian-orbit-distribution}
For $t\ge0$,
\begin{equation}
 |\Fix(\pi_q^t)|
 =\prod_{j=1}^s\gcd(n_j,q^t-1).                 \label{eq:abelian-fixed-points}
\end{equation}
For every $e\mid m$, the number $A_e(A)$ of characters whose Frobenius orbit
has length exactly $e$, and the number $B_e(A)$ of such orbits, are
\begin{align}
 A_e(A)&=\sum_{r\mid e}\mu(e/r)
          \prod_{j=1}^s\gcd(n_j,q^r-1),          \label{eq:abelian-exact-elements}\\
 B_e(A)&=\frac1e A_e(A).                          \label{eq:abelian-exact-orbits}
\end{align}
\end{theorem}

\begin{proof}
After identifying $\widehat A$ with the product in
\eqref{eq:abelian-decomposition}, $\pi_q^t$ is componentwise multiplication
by $q^t$.  The congruence $(q^t-1)x=0$ in $\mathbb Z/n_j\mathbb Z$ has
$\gcd(n_j,q^t-1)$ solutions, proving \eqref{eq:abelian-fixed-points}.
Characters with orbit length dividing $e$ are exactly the fixed points of
$\pi_q^e$.  M\"obius inversion gives
\eqref{eq:abelian-exact-elements}; division by $e$ gives
\eqref{eq:abelian-exact-orbits}.
\end{proof}

For $A=\mathbb Z/n\mathbb Z$, where $n\mid Q-1$, these are the usual
$q$-cyclotomic classes modulo $n$, and
\begin{equation}
 B_e(n)=\frac1e\sum_{r\mid e}\mu(e/r)\gcd(n,q^r-1).
                                                        \label{eq:cyclic-Be}
\end{equation}

In the full cyclic case $n=q^m-1$, for every $e\mid m$,
\begin{equation}
 B_e(q^m-1)
 =\frac1e\sum_{r\mid e}\mu(e/r)(q^r-1).          \label{eq:full-Be}
\end{equation}
These formulas are classical orbit counts, equivalent to the enumeration of
nonzero elements of prescribed degree over $\K$; see
\cite{LidlNiederreiter,HuffmanPless}.

\section{Exact distance and coset geometry}
\label{sec:geometry}

\subsection{Distance enumerators and lists}

Let $g\in\Lfield^I$ be arbitrary.  For a cycle $O$ of length $\ell_O$, define
the normalized received symbols
\begin{equation}
 x_{O,t}=g_{i_{O,t}}^{q^{m-t}},
 \qquad 0\le t<\ell_O,                           \label{eq:received-normalization}
\end{equation}
and, for $b\in S_O=\F{q^{\ell_O}}$, their multiplicities
\begin{equation}
 \nu_O(b)=|\{t:x_{O,t}=b\}|,
 \qquad
 M_O=\max_{b\in S_O}\nu_O(b).                   \label{eq:cycle-multiplicity}
\end{equation}
Values $x_{O,t}\notin S_O$ do not match any admissible seed.

\begin{theorem}
\label{thm:distance-enumerator}
The polynomial enumerating all distances from $g$ to $\Cpi$ factors as
\begin{equation}
 D_g(z):=\sum_{f\in\Cpi}z^{\dist(g,f)}
 =\prod_{O\in\mathfrak O(\pi)}
   \left(\sum_{b\in S_O}z^{\ell_O-\nu_O(b)}\right).       \label{eq:distance-enumerator}
\end{equation}
Consequently,
\begin{equation}
 \dist(g,\Cpi)
 =\sum_{O\in\mathfrak O(\pi)}(\ell_O-M_O),       \label{eq:nearest-distance}
\end{equation}
and the number of nearest codewords is
\begin{equation}
 \prod_{O\in\mathfrak O(\pi)}
 \left|\{b\in S_O:\nu_O(b)=M_O\}\right|.          \label{eq:number-nearest}
\end{equation}
\end{theorem}

\begin{proof}
On $O$, the codeword with seed $b$ has coordinates $b^{q^t}$.  By applying
the inverse Frobenius to both symbols,
\[
 g_{i_{O,t}}=b^{q^t}
 \quad\Longleftrightarrow\quad
 x_{O,t}=b.
\]
Hence the local distance is $\ell_O-\nu_O(b)$.  The seed choices on distinct
cycles are independent, so summing $z^{\dist(g,f)}$ over all codewords gives
\eqref{eq:distance-enumerator}.  Its smallest exponent and the coefficient
of that exponent give \eqref{eq:nearest-distance} and
\eqref{eq:number-nearest}.
\end{proof}

\begin{remark}
Formula \eqref{eq:nearest-distance} gives a linear-size modal decoder after
the normalized values and subfield tests are available.  If no normalized
value on $O$ lies in $S_O$, then $M_O=0$ and all $q^{\ell_O}$ seeds are
equally near.  Thus choosing the zero seed is permissible but is not unique.
\end{remark}

\begin{corollary}
\label{cor:exact-list-size}
For an integer $R\ge0$,
\begin{equation}
 |\{f\in\Cpi:\dist(g,f)\le R\}|
 =\sum_{j=0}^{R}[z^j]D_g(z).                    \label{eq:global-list-size}
\end{equation}
On one cycle of length $\ell$, the list size at an integer local radius
$r\ge0$ is
\begin{equation}
 L_O(g,r)=
 \left|\{b\in\F{q^\ell}:\nu_O(b)\ge\ell-r\}\right|.       \label{eq:local-list-size}
\end{equation}
For $0\le r<\ell$, the exact worst-case value is
\begin{equation}
 \max_g L_O(g,r)
 =\left\lfloor\frac{\ell}{\ell-r}\right\rfloor,           \label{eq:worst-list-size}
\end{equation}
whereas for $r\ge\ell$ it is $q^\ell$.
\end{corollary}

\begin{proof}
Equation \eqref{eq:global-list-size} follows from
Theorem~\ref{thm:distance-enumerator}, and
\eqref{eq:local-list-size} is the one-cycle specialization.  If
$k=\ell-r\ge1$, then distinct seeds counted by the list occupy disjoint sets
of at least $k$ coordinates, so their number is at most
$\lfloor\ell/k\rfloor$.  Equality is obtained by placing $k$ copies of each
of $\lfloor\ell/k\rfloor$ distinct admissible values and filling any remaining
positions with already used values; this creates no additional seed of
multiplicity at least $k$.  There are enough admissible values because
$q^\ell\ge\ell$.
\end{proof}

\subsection{Covering radius}

\begin{theorem}
\label{thm:covering-radius}
For a cycle of length $\ell\mid m$, the constituent
\[
 C_{\ell,m}
 =\{(b,b^q,\ldots,b^{q^{\ell-1}}):b\in\F{q^\ell}\}
 \subseteq\Lfield^\ell
\]
has symbol covering radius
\begin{equation}
 \rho_{\ell,m}=
 \begin{cases}
  \ell,&\ell<m,\\
  \ell-1,&\ell=m.
 \end{cases}                                      \label{eq:local-covering-radius}
\end{equation}
If $B_m(\pi)$ denotes the number of $m$-cycles of $\pi$, then
\begin{equation}
 \rho_{\mathrm H}(\Cpi)=N-B_m(\pi).              \label{eq:global-covering-radius}
\end{equation}
For Fourier spectra on a finite abelian group $A$, one has
$B_m(\pi_q)=B_m(A)$ from Theorem~\ref{thm:abelian-orbit-distribution}.
\end{theorem}

\begin{proof}
Use the diagonal normal form.  If $\ell<m$, the seed field
$S=\F{q^\ell}$ is a proper subfield of $\Lfield$.  Choosing every normalized
coordinate outside $S$ gives distance $\ell$ from every codeword.  If
$\ell=m$, then $S=\Lfield$, so every received block has at least one match
with a constant word.  Conversely, $\ell$ pairwise distinct normalized
symbols exist because $Q=q^m\ge m$ and give distance $\ell-1$.

The code and the ambient Hamming space are direct products over the cycles.
Covering radii add under Cartesian products, and summing
\eqref{eq:local-covering-radius} gives
\[
 \sum_O\ell_O-|\{O:\ell_O=m\}|=N-B_m(\pi).
\]
\end{proof}

\begin{remark}
In the Fourier specialization, if the order of the Frobenius permutation is
strictly less than $m$, then $B_m(A)=0$ and the symbol covering radius equals
the full length $N$.
\end{remark}

\subsection{Complete coset-leader distribution}

For integers $\ell\ge1$ and $-1\le k\le\ell$, put
\begin{equation}
 \mathcal A_{\ell,m}(k)=
 \begin{cases}
 0,&k=-1,\\[1mm]
 \displaystyle
 \ell!\,[u^\ell]\,
 \exp\bigl((Q-q^\ell)u\bigr)
 \left(\sum_{j=0}^{k}\frac{u^j}{j!}\right)^{q^\ell},
 &0\le k\le\ell.
 \end{cases}                                      \label{eq:occupancy-A}
\end{equation}
Here $[u^\ell]$ denotes coefficient extraction in the formal power series.

\begin{theorem}
\label{thm:coset-leader-enumerator}
For a coset $\Gamma\in\Lfield^\ell/C_{\ell,m}$, put
\[
 d(\Gamma,C_{\ell,m})=\min_{y\in\Gamma}\wtH(y).
\]
Define
\[
 \Lambda_{\ell,m}(z)
 =\sum_{\Gamma\in\Lfield^\ell/C_{\ell,m}}
   z^{d(\Gamma,C_{\ell,m})}.
\]
Then
\begin{equation}
 \Lambda_{\ell,m}(z)
 =q^{-\ell}\sum_{k=0}^{\ell}
   \bigl(\mathcal A_{\ell,m}(k)
        -\mathcal A_{\ell,m}(k-1)\bigr)z^{\ell-k}.        \label{eq:local-coset-enumerator}
\end{equation}
The global coset-leader enumerator is
\begin{equation}
 \Lambda_{\Cpi}(z)
 =\prod_{O\in\mathfrak O(\pi)}
   \Lambda_{\ell_O,m}(z).                        \label{eq:global-coset-enumerator}
\end{equation}
In particular, the coefficient of $z^r$ is the exact number of cosets of
$\Cpi$ having coset-leader weight $r$.
\end{theorem}

\begin{proof}
After normalization, let $N_b$ be the number of occurrences of
$b\in S=\F{q^\ell}$ in an ambient block, and put
$M=\max_{b\in S}N_b$.  The local distance is $\ell-M$.
The number of labelled words for which $M\le k$ is
$\mathcal A_{\ell,m}(k)$.  Indeed, each of the $q^\ell$ admissible symbols
has exponential generating function
$\sum_{j=0}^{k}u^j/j!$, while each of the $Q-q^\ell$ outside symbols has no
multiplicity restriction and contributes $e^u$.  Thus the number of words
with $M=k$ is
$\mathcal A_{\ell,m}(k)-\mathcal A_{\ell,m}(k-1)$.

Every coset contains $|C_{\ell,m}|=q^\ell$ words, and distance to an additive
code is constant on a coset.  Division by $q^\ell$ proves
\eqref{eq:local-coset-enumerator}.  Both the quotient and the distance split
over disjoint cycles, proving \eqref{eq:global-coset-enumerator}.
\end{proof}

\begin{corollary}
\label{cor:random-distance-distribution}
If $G_O$ is uniform in $\Lfield^\ell$, then
\begin{equation}
 \Pr\{\dist(G_O,C_{\ell,m})=\ell-k\}
 =\frac{\mathcal A_{\ell,m}(k)
        -\mathcal A_{\ell,m}(k-1)}{Q^\ell}.       \label{eq:local-random-distance}
\end{equation}
For uniform $G\in\Lfield^I$, the distances on different cycles are
independent and the probability generating function of
$\dist(G,\Cpi)$ is
\begin{equation}
 \mathbb E\,z^{\dist(G,\Cpi)}
 =\prod_{O\in\mathfrak O(\pi)}
 Q^{-\ell_O}
 \sum_{k=0}^{\ell_O}
 \bigl(\mathcal A_{\ell_O,m}(k)
       -\mathcal A_{\ell_O,m}(k-1)\bigr)z^{\ell_O-k}.
 \label{eq:global-random-pgf}
\end{equation}
\end{corollary}

\begin{proof}
Divide the word counts in the proof of
Theorem~\ref{thm:coset-leader-enumerator} by $Q^\ell$.  Independence follows
from the disjoint coordinate blocks.
\end{proof}

\begin{corollary}
\label{cor:mean-unique-nearest}
For a uniform block $G_O\in\Lfield^\ell$,
\begin{equation}
 \mathbb E\,\dist(G_O,C_{\ell,m})
 =Q^{-\ell}\sum_{k=0}^{\ell-1}\mathcal A_{\ell,m}(k).  \label{eq:local-mean-distance}
\end{equation}
Consequently, for uniform $G\in\Lfield^I$,
\begin{equation}
 \mathbb E\,\dist(G,\Cpi)
 =\sum_{O\in\mathfrak O(\pi)}Q^{-\ell_O}
   \sum_{k=0}^{\ell_O-1}\mathcal A_{\ell_O,m}(k).        \label{eq:global-mean-distance}
\end{equation}
Put $s=q^\ell$ and
\begin{equation}
 \mathcal U_{\ell,m}
 =s\sum_{k=1}^{\ell}\binom{\ell}{k}(\ell-k)!
   [u^{\ell-k}]\,
   e^{(Q-s)u}
   \left(\sum_{j=0}^{k-1}\frac{u^j}{j!}\right)^{s-1}.
                                                        \label{eq:local-unique-count}
\end{equation}
Then the probability that $G_O$ has a unique nearest word in
$C_{\ell,m}$ is $\mathcal U_{\ell,m}/Q^\ell$.  The probability that a
uniform $G\in\Lfield^I$ has a unique nearest word in $\Cpi$ is
\begin{equation}
 \prod_{O\in\mathfrak O(\pi)}
   \frac{\mathcal U_{\ell_O,m}}{Q^{\ell_O}}.             \label{eq:global-unique-probability}
\end{equation}
\end{corollary}

\begin{proof}
With $M=\max_{b\in\F{q^\ell}}N_b$, the tail-sum identity gives
\[
 \mathbb E M
 =\sum_{j=1}^{\ell}\Pr\{M\ge j\}
 =\ell-Q^{-\ell}\sum_{k=0}^{\ell-1}\mathcal A_{\ell,m}(k).
\]
Since the local distance is $\ell-M$, this proves
\eqref{eq:local-mean-distance}; additivity over the cycles proves
\eqref{eq:global-mean-distance}.

For uniqueness, choose the unique winning seed in $s$ ways and suppose that
it occurs exactly $k\ge1$ times.  Choose its positions in
$\binom{\ell}{k}$ ways.  Among the remaining labelled positions, every one
of the other $s-1$ admissible symbols may occur at most $k-1$ times, whereas
the $Q-s$ outside symbols are unrestricted.  The corresponding exponential
generating function is the coefficient in
\eqref{eq:local-unique-count}.  Summing over $k$ counts every block with a
unique modal seed exactly once.  Finally, nearest words and received blocks
factor over the cycles, which proves \eqref{eq:global-unique-probability}.
\end{proof}

\begin{corollary}
\label{cor:deep-holes}
The number $H_{\ell,m}$ of ambient deep holes of the local constituent is
\begin{equation}
 H_{\ell,m}=
 \begin{cases}
 (Q-q^\ell)^\ell,&\ell<m,\\
 (Q)_m:=Q(Q-1)\cdots(Q-m+1),&\ell=m.
 \end{cases}                                      \label{eq:local-deep-holes}
\end{equation}
The corresponding number of deepest local cosets is
$H_{\ell,m}/q^\ell$.  Globally, the numbers of deep holes and deepest cosets
are respectively
\begin{equation}
 \prod_{O}H_{\ell_O,m},
 \qquad
 \prod_O\frac{H_{\ell_O,m}}{q^{\ell_O}}.          \label{eq:global-deep-holes}
\end{equation}
\end{corollary}

\begin{proof}
For $\ell<m$, distance $\ell$ means that none of the normalized symbols lies
in $\F{q^\ell}$.  For $\ell=m$, distance $m-1$ means that the largest
multiplicity is one, i.e. all normalized symbols are distinct.  This gives
\eqref{eq:local-deep-holes}.  Products follow from the direct decomposition.
\end{proof}

\section{The full cyclic family and covering asymptotics}
\label{sec:cyclic}

Let
\[
 n=Q-1=q^m-1
\]
and let $\pi(s)=qs$ on $\mathbb Z/n\mathbb Z$.  We write
$F_{q,m}=\mathcal C(\pi)$.  This is the Frobenius-consistent spectral space
of the length-$n$ Fourier transform.  For brevity, write
$B_e=B_e(q^m-1)$, and let $\tau(r)$ denote the number of positive divisors
of $r$.  By
\eqref{eq:full-Be} and Theorem~\ref{thm:covering-radius},
\begin{equation}
 \rho_{\mathrm H}(F_{q,m})
 =n-\frac1m\sum_{r\mid m}\mu(m/r)(q^r-1).        \label{eq:full-covering-exact}
\end{equation}

For $0\le\delta\le1-1/Q$, define the $Q$-ary entropy function
\begin{equation}
 H_Q(\delta)
 =\delta\log_Q(Q-1)
  -\delta\log_Q\delta
  -(1-\delta)\log_Q(1-\delta),                  \label{eq:qary-entropy}
\end{equation}
with the usual convention $0\log 0=0$.  It is strictly increasing on the
displayed interval.  Let
\[
 \delta^*_{q,m}=H_Q^{-1}(1-1/m).
\]
In the asymptotic expansions below, $\log$ without a displayed base denotes
the natural logarithm.

\begin{theorem}
\label{thm:covering-efficiency}
Fix $q$ and let $m\to\infty$.  Then
\begin{equation}
 \frac{\rho_{\mathrm H}(F_{q,m})}{q^m-1}
 =1-\frac1m
  +O\left(\frac{\tau(m)}{m}q^{-m/2}\right).      \label{eq:normalized-radius-asymptotic}
\end{equation}
Every nonempty code $C\subseteq\Lfield^n$ of cardinality at most
$q^n=Q^{n/m}$ satisfies the rate-only sphere-covering bound
\begin{equation}
 \frac{\rho_{\mathrm H}(C)}{n}\ge\delta^*_{q,m}. \label{eq:sphere-covering-inverse}
\end{equation}
Moreover,
\begin{equation}
 \delta^*_{q,m}
 =1-\frac1m-\frac{\log m+1}{m^2\log q}
  +O\left(\frac{(\log m)^2}{m^3}\right).         \label{eq:entropy-inverse-asymptotic}
\end{equation}
Consequently,
\begin{equation}
 \frac{\rho_{\mathrm H}(F_{q,m})}{q^m-1}-\delta^*_{q,m}
 =\frac{\log m+1}{m^2\log q}
  +O\left(\frac{(\log m)^2}{m^3}\right).         \label{eq:covering-gap-second-order}
\end{equation}
\end{theorem}

\begin{proof}
Formula \eqref{eq:full-Be} with $e=m$ gives
\[
 B_m(q^m-1)
 =\frac{q^m-1}{m}
  +O\left(\frac{\tau(m)}m q^{m/2}\right),
\]
because every proper divisor of $m$ is at most $m/2$.  Combining this with
\eqref{eq:global-covering-radius} proves
\eqref{eq:normalized-radius-asymptotic}.

Let $R=\rho_{\mathrm H}(C)$.  The Hamming balls of radius $R$ around the
codewords cover $\Lfield^n$, so
\[
 |C|\sum_{j=0}^{R}\binom nj(Q-1)^j\ge Q^n.
\]
If $R/n\le1-1/Q$, the standard entropy estimate for a $Q$-ary Hamming ball
gives
\[
 \sum_{j=0}^{R}\binom nj(Q-1)^j
 \le Q^{nH_Q(R/n)}.
\]
Since $|C|\le Q^{n/m}$, it follows that
$H_Q(R/n)\ge1-1/m$, proving
\eqref{eq:sphere-covering-inverse}.  If $R/n>1-1/Q$, the same conclusion is
automatic.  This is the usual sphere-covering argument; see~\cite{CohenEtAl}.

For the sharper inverse-entropy expansion, put $\varepsilon=1-\delta$ and
write
\[
 h(\varepsilon)
 =-\varepsilon\log\varepsilon-(1-\varepsilon)\log(1-\varepsilon).
\]
Uniformly for $0\le\varepsilon\le1$,
\begin{equation}
 H_Q(1-\varepsilon)
 =1-\varepsilon+\frac{h(\varepsilon)}{m\log q}
  +O\left(\frac1{Q\log Q}\right).               \label{eq:entropy-local-expansion}
\end{equation}
Let $\varepsilon^*=1-\delta^*_{q,m}$.  Substitution in
\eqref{eq:entropy-local-expansion} gives the fixed-point relation
\begin{equation}
 \varepsilon^*=\frac1m+\frac{h(\varepsilon^*)}{m\log q}
  +O\left(\frac1{Q\log Q}\right).                \label{eq:entropy-fixed-point}
\end{equation}
Using first $0\le h\le\log2$ and then
$h(\varepsilon)=O(\varepsilon\log(1/\varepsilon))$ gives successively
\[
 \varepsilon^*=\Theta(1/m),\qquad
 \varepsilon^*=\frac1m+O\left(\frac{\log m}{m^2}\right).
\]
Since $h'(x)=\log((1-x)/x)$,
\[
 h(\varepsilon^*)
 =h(1/m)+O\left(\frac{(\log m)^2}{m^2}\right)
 =\frac{\log m+1}{m}+O\left(\frac{(\log m)^2}{m^2}\right).
\]
Substitution once more in \eqref{eq:entropy-local-expansion} proves
\eqref{eq:entropy-inverse-asymptotic}.  Finally, the exponentially small
error in \eqref{eq:normalized-radius-asymptotic} is absorbed by the remainder
in \eqref{eq:entropy-inverse-asymptotic}; their difference is
\eqref{eq:covering-gap-second-order}.
\end{proof}

\begin{remark}
The asymptotic family in Theorem~\ref{thm:covering-efficiency} has changing
alphabet $Q=q^m$ and length $n=Q-1$.  This explicit regime replaces the
ambiguous $o(1)$ that would result from sending $n\to\infty$ while keeping
$q$ and $m$ fixed; under the Fourier splitting condition, such a fixed-field
limit is unavailable because $n\le q^m-1$.
\end{remark}

\begin{theorem}
\label{thm:poisson-deficit}
Let $m\ge2$, let $G_m$ be uniform in $\Lfield^{q^m-1}$, and put
\[
 X_m=\rho_{\mathrm H}(F_{q,m})-\dist(G_m,F_{q,m}),
 \qquad
 \lambda_m=\frac{m-1}{2}+\mathbf 1_{\{2\mid m\}}.
\]
For fixed $q$ and $m\to\infty$,
\begin{equation}
 d_{\mathrm{TV}}\bigl(\mathcal L(X_m),
        \operatorname{Poisson}(\lambda_m)\bigr)\longrightarrow0,            \label{eq:poisson-tv}
\end{equation}
where $d_{\mathrm{TV}}$ is total variation distance.  Moreover,
\begin{align}
 \mathbb E X_m&=\lambda_m+o(1),                                   \label{eq:deficit-mean}\\
 \frac{X_m-\lambda_m}{\sqrt{\lambda_m}}
   &\xrightarrow{\mathrm d}\mathcal N(0,1),                       \label{eq:deficit-clt}\\
 \log\Pr\{\dist(G_m,F_{q,m})=\rho_{\mathrm H}(F_{q,m})\}
   &=-\lambda_m+o(1).                                               \label{eq:deep-hole-asymptotic}
\end{align}
In particular, the probability that a uniform word is a deep hole equals
$\exp(-\lambda_m+o(1))$.
\end{theorem}

\begin{proof}
The cycle blocks are independent.  On a cycle of length $e<m$, let $M_e$
be the largest occupancy of an admissible symbol in $\F{q^e}$; its
contribution to $X_m$ is $M_e$.  On a full cycle the contribution is
$M_m-1$.

First consider the $B_m$ full cycles.  For one such cycle put
$J=\mathbf 1_{\{M_m\ge2\}}$.  The variables $J$ are independent Bernoulli
variables with
\begin{equation}
 p_m=1-\frac{(Q)_m}{Q^m}
     =\frac{\binom m2}{Q}+O\left(\frac{m^4}{Q^2}\right).             \label{eq:full-collision-probability}
\end{equation}
The variables $M_m-1$ and $J$ differ only if some symbol occurs at least
three times.  A union bound over triples gives
\begin{equation}
 \Pr\left\{\sum_{i=1}^{B_m}(M_{m,i}-1)
       \ne\sum_{i=1}^{B_m}J_i\right\}
 \le B_m\binom m3 Q^{-2}=O(m^2/Q).                \label{eq:triple-collision-bound}
\end{equation}
By \eqref{eq:full-Be},
\begin{equation}
 B_mp_m=\frac{m-1}{2}+o(1),
 \qquad B_mp_m^2=O(m^3/Q)=o(1).                  \label{eq:full-poisson-parameters}
\end{equation}
Le Cam's inequality~\cite{LeCam1960} therefore approximates
$\sum_iJ_i$ in total variation by a Poisson variable of mean
$(m-1)/2+o(1)$.

If $m$ is even, put $e=m/2$ and $s=q^e$.  On an $e$-cycle the indicator
$H=\mathbf 1_{\{M_e\ge1\}}$ is Bernoulli with parameter
\begin{equation}
 r_m=1-(1-s/Q)^e=1-(1-s^{-1})^e.                  \label{eq:half-hit-probability}
\end{equation}
Formula \eqref{eq:full-Be}, applied with $e$, and the binomial expansion in
\eqref{eq:half-hit-probability} give
\[
 B_e=\frac{s}{e}+O\left(\frac{\tau(e)}e q^{e/2}\right),
 \qquad
 r_m=\frac{e}{s}+O\left(\frac{e^2}{s^2}\right).
\]
Hence
\begin{equation}
 B_er_m=1+O\left(\frac es+\tau(e)q^{-e/2}\right)\longrightarrow1,
 \qquad B_er_m^2=O(e/s)\longrightarrow0.          \label{eq:half-poisson-parameters}
\end{equation}
Furthermore,
\[
 B_e\Pr\{M_e\ge2\}
 \le B_e\binom e2\frac{s}{Q^2}=o(1),
\]
because two positions occupied by the same admissible symbol have
probability $s/Q^2$.  Hence the total contribution of the half cycles is,
in total variation, asymptotically Poisson with mean $1$.  This contribution
is independent of that of the full cycles.

Every remaining proper divisor satisfies $e\le m/3$.  Since
$eB_e\le q^e-1$, a union bound gives
\begin{equation}
 \Pr\{\text{some remaining cycle contributes to }X_m\}
 \le\sum_{\substack{e\mid m\\e\le m/3}}
       B_e e q^{e-m}
 \le\tau(m)q^{-m/3}=o(1),                        \label{eq:short-cycle-negligible}
\end{equation}
where we used $\tau(m)\le m$.  The sum of
independent Poisson variables is Poisson.  Moreover, total variation does not
increase under convolution, and
\[
 d_{\mathrm{TV}}\bigl(\operatorname{Poisson}(\mu),
   \operatorname{Poisson}(\nu)\bigr)\le|\mu-\nu|.
\]
Consequently, combining
\eqref{eq:triple-collision-bound}--\eqref{eq:short-cycle-negligible} proves
\eqref{eq:poisson-tv}.

The same bounds also control the first moment.  On a full cycle and, when
$m$ is even, on a half cycle, respectively,
\[
 0\le(M_m-1)-J\le m\mathbf 1_{\{M_m\ge3\}},
 \qquad
 0\le M_e-H\le e\mathbf 1_{\{M_e\ge2\}}.
\]
Thus their total expectation errors are, respectively,
$O(m^3/Q)$ and $O(e^2/s^2)$, both $o(1)$.  On every remaining cycle,
$M_e$ is at most the number of admissible hits, so its total expected
contribution is bounded by the right-hand side of
\eqref{eq:short-cycle-negligible}.  Together with
\eqref{eq:full-poisson-parameters} and
\eqref{eq:half-poisson-parameters}, this proves \eqref{eq:deficit-mean}.
Since $\lambda_m\to\infty$, the classical central limit theorem for a
Poisson variable, combined with \eqref{eq:poisson-tv}, gives
\eqref{eq:deficit-clt}.

It remains to estimate the atom at zero directly, since total variation
alone does not give its relative asymptotics.  Corollary~\ref{cor:deep-holes}
gives
\begin{align*}
 \log\Pr\{X_m=0\}
 &=B_m\sum_{j=0}^{m-1}\log(1-j/Q)\\
 &\quad+\sum_{\substack{e\mid m\\e<m}}
       B_e e\log(1-q^{e-m}).
\end{align*}
Using $\log(1-x)=-x+O(x^2)$, the full-cycle term is
\[
 -B_m\frac{\binom m2}{Q}
 +O\left(\frac{B_m m^3}{Q^2}\right)
 =-\frac{m-1}{2}+o(1).
\]
When $m$ is even, the half-cycle term satisfies
\[
 B_e e\log(1-1/s)
 =-\frac{eB_e}{s}+O\left(\frac{eB_e}{s^2}\right)
 =-1+o(1).
\]
For $e\le m/3$, the inequality $|\log(1-x)|\le2x$ for sufficiently small
$x$ and the estimate in \eqref{eq:short-cycle-negligible} show that the
remaining terms sum to $o(1)$.  This proves
\eqref{eq:deep-hole-asymptotic}.
\end{proof}

\begin{example}
\label{ex:binary-seven}
Take $q=2$, $m=3$, $Q=8$, and $n=7$.  The cyclotomic classes are
\[
 \{0\},\qquad \{1,2,4\},\qquad \{3,6,5\}.
\]
Hence
\begin{align*}
 W_{F_{2,3}}(z)&=(1+z)(1+7z^3)^2,\\
 d_{\mathrm H}(F_{2,3})&=1,\qquad
 \rho_{\mathrm H}(F_{2,3})=7-2=5.
\end{align*}
For the one-cycle constituent,
$\Lambda_{1,3}(z)=1+3z$.  A full three-cycle is an $8$-ary repetition code:
its $64$ cosets consist of one coset at distance $0$, $21$ at distance $1$,
and $42$ at distance $2$.  Therefore
\begin{equation}
 \Lambda_{F_{2,3}}(z)
 =(1+3z)(1+21z+42z^2)^2.                        \label{eq:example-coset-enumerator}
\end{equation}
\end{example}

\section{Conclusion}

The standard Frobenius conjugacy constraint on finite-field Fourier spectra
defines a subfield-linear code whose symbol-Hamming geometry is controlled by
the Frobenius cycle lengths.  The diagonal isometry to a product of subfield
repetition codes gives exact received-word and list enumerators, covering
radii, coset-leader distributions, mean distances, unique-nearest-word
probabilities, and deep-hole counts.  The Fourier descent and orbit formulas
used to identify the specialization are classical; the contribution is the
resulting exact metric analysis.

For the full cyclic family, the same occupancy model shows that the random
deficit from the covering radius is asymptotically Poisson and hence satisfies
a central limit theorem.  It also yields the deep-hole probability and the
second-order gap between the exact radius and the rate-only sphere-covering
bound.  Since the trivial character forces minimum distance one, these spaces
are naturally nearest-word approximation spaces rather than strong
worst-case error-correcting codes.

\end{document}